\documentclass[a4,12pt]{amsart}

\usepackage[T1]{fontenc}
\usepackage[utf8]{inputenc}
\usepackage{amsmath,amssymb,amsthm,mathtools}
\usepackage{enumitem}
\usepackage[hidelinks]{hyperref}

\newtheorem{theorem}{Theorem}[section]
\newtheorem{proposition}[theorem]{Proposition}
\newtheorem{lemma}[theorem]{Lemma}
\newtheorem{corollary}[theorem]{Corollary}
\newtheorem{conjecture}[theorem]{Conjecture}
\newtheorem*{claim}{Claim}
\theoremstyle{definition}
\newtheorem{definition}[theorem]{Definition}
\newtheorem{setting}[theorem]{Setting}
\newtheorem{example}[theorem]{Example}
\newtheorem{remark}[theorem]{Remark}

\numberwithin{equation}{section}

\newcommand{\NN}{\mathbb N}
\newcommand{\ZZ}{\mathbb Z}
\newcommand{\PF}{\operatorname{PF}}
\newcommand{\Ap}{\operatorname{Ap}}
\newcommand{\Ker}{\operatorname{Ker}}
\newcommand{\htop}{\operatorname{ht}}

\newcommand{\Zfac}{\mathsf Z}
\newcommand{\X}{\mathsf X}
\newcommand{\rmI}{\operatorname{I}}

\title[Arithmetic pseudo-Frobenius numbers]
{ARITHMETIC PSEUDO-FROBENIUS NUMBERS AND DETERMINANTAL NUMERICAL SEMIGROUP RINGS}

\author{DO VAN KIEN}

\address{Department of Mathematics, Hanoi Pedagogical University 2, Xuan Hoa, Phu Tho, Vietnam}
\email{dovankien@hpu2.edu.vn}


\begin{document}

\begin{abstract}
Let \(H=\langle a_1,\ldots,a_n\rangle\) be a numerical semigroup. A conjecture on numerical semigroup rings predicts that the defining ideal \(I_H\) has a determinantal presentation if and only if the set of pseudo-Frobenius numbers of \(H\) has form $\{h+\alpha,h+2\alpha,\ldots,h+(n-1)\alpha\}$ for $h\ge 0, \alpha>0$. We prove this conjecture under the condition that \(h\) has a unique factorization in \(H\). We also give characterization this
unique-factorization condition. As an application, we consider the family of affine-orbit numerical semigroups determined by \(m=cA_{n-1}+1\), \(1\le c\le a\). We prove that the relevant element \(h\) has a unique factorization and obtain
\[
I_H=
\rmI_2\begin{pmatrix}
X_1^a&x_2^a&\cdots&X_{n-1}^a&X_n^c\\
X_2&x_3&\cdots&X_n&X_1^{d+c+1-a}
\end{pmatrix}.
\]
In particular, the Eagon--Northcott complex of the matrix mentioned above gives the graded minimal free resolution of \(k[H]\).
\end{abstract}

\maketitle

\section{Introduction}

Let $H=\langle a_1,\ldots,a_n\rangle$
be a numerical semigroup minimally generated by \(n\ge3\) positive integers, and let $k[H]=k[t^{a_1},\ldots,t^{a_n}]$
be its semigroup ring over a field \(k\). Let
$S=k[X_1,\ldots,X_n]$ be the weighted polynomial ring with $	\deg X_i=a_i$ for every $1\le i\le n$ and consider the graded \(k\)-algebra homomorphism $\varphi:S\longrightarrow k[H]$ given by
$\varphi(X_i)=t^{a_i}$ for all $i$. The ideal $I_H=\Ker\varphi$ is called the {\it defining ideal} of the numerical semigroup ring \(k[H]\). Determining the structure of \(I_H\) is a classical problem in the study of monomial curves and numerical semigroup rings. There are a few known results about the structure of $I_H$. The most important result is due to J. Herzog \cite{Herzog}  when $n = 3$, but for semigroups with
more generators, the situation becomes significantly more complex. When $n=4$, there are some partial answers to describe the structure of the defining ideal: e.g. symmetric semigroups (\cite{Bresinsky}), pseudo-symmetric semigroups (\cite{Komeda}), and almost symmetric semigroups (\cite{Eto}). For $n(\ge 4)$-generated cases, research has primarily focused on specific classes, such as those generated by arithmetic sequences \cite{GSS} or repunit numbers \cite{BCO, RBT}, making a comprehensive understanding still unattainable.

In general, explicitly determining $I_H$ presents substantial difficulty, even when its numerical invariants have been established. A particularly useful numerical invariant is the set of pseudo-Frobenius numbers
\[
\PF(H)=
\{f\in\ZZ\setminus H\mid f+s\in H\text{ for every }0<s\in H\}.
\]
These numbers determine the degrees of the minimal generators of the canonical module of \(k[H]\); in particular,  the Cohen-Macaulay type $\rm{r}(k[H])$ of $k[H]$ is equal to $\sharp \PF(H)$ (see \cite{GW}). Motivated by the interaction between pseudo-Frobenius numbers and defining ideals established in \cite{GKMT,KM}, the following
conjecture was formulated by the author together with D. T. Cuong, N. Matsuoka, and H. L Truong  (see \cite[Conjecture~1.1]{KMO}).

\begin{conjecture}\label{conj:main}
Let \(H=\langle a_1,\ldots,a_n\rangle\) be minimally generated by \(n\ge3\)
elements. The following conditions are equivalent.
\begin{enumerate}
\item[(1)] After a suitable permutation of the generators,
\[
I_H=
\rmI_2\begin{pmatrix}
X_1^{\ell_1}&X_2^{\ell_2}&\cdots&X_n^{\ell_n}\\
X_2^{m_2}&X_3^{m_3}&\cdots&X_1^{m_1}
\end{pmatrix}
\]
for some positive integers
\(\ell_1,\ldots,\ell_n,m_1,\ldots,m_n\).
\item [(2)] There exist \(h\ge0\) and \(\alpha>0\) such that
\[
\PF(H)=\{h+\alpha,h+2\alpha,\ldots,h+(n-1)\alpha\}.
\]
\end{enumerate}
\end{conjecture}
\noindent Here, $\rmI_2(M)$ denotes the ideal of $S$ generated by $2\times 2$ minors of a matrix $M$ whose entries are in $S$.

The implication \((1)\Rightarrow(2)\) has already been established in \cite{KM}. The converse remains open in general, although it has been proved for several classes, including almost
symmetric numerical semigroups \cite{GKMT}, numerical semigroups of maximal embedding dimension \cite{KM}, stretched numerical semigroup rings \cite{KMO}, generalized repunit numerical semigroups \cite{BCO,CO}, and certain families satisfying restrictions involving the multiplicity and the embedding
dimension \cite{Takahashi}.

The main result of this paper gives a general sufficient condition for the
converse. Assume that
\[
\PF(H)=\{h+\alpha,h+2\alpha,\ldots,h+(n-1)\alpha\}.
\]
For \(1\le i\le n\), put
\[
m_i=\min\{\ell>0\mid\ell a_i+\alpha\in H\},\, M_i=m_i-1,
\]
and set
\[
h'=M_1 a_1+M_2 a_2+\cdots+M_n a_n.
\]
We first show that
\[
M_i=
\max\{q_i\mid (q_1,\ldots,q_n)\in\Zfac(h)\},
\]
where \(\Zfac(h)\) denotes the set of factorizations of \(h\). This yields the equivalences
\[
h\text{ has a unique factorization}
\quad\Longleftrightarrow\quad
h'=h
\quad\Longleftrightarrow\quad
h'+\alpha\notin H.
\]
Thus a homological problem concerning \(I_H\) is reduced to the factorization of a single element of \(H\).

The central theorem is the following.

\begin{theorem}[Theorem \ref{thm:general}]\label{thm:intro-general}
Suppose that $\PF(H)=\{h+\alpha,h+2\alpha,\ldots,h+(n-1)\alpha\}$
and that $h=q_1a_1+\cdots+q_na_n$ is the unique factorization of \(h\) in $H$ . Then, after a suitable permutation of
\(a_1,\ldots,a_n\), there exist positive integers
\(r_1,\ldots,r_n\) such that
\[
(q_i+1)a_i+\alpha=r_{i+1}a_{i+1}, 1\le\forall i\le n,
\]
with indices taken modulo \(n\). Consequently,
\[
I_H=
\rmI_2\begin{pmatrix}
X_1^{q_1+1}&X_2^{q_2+1}&\cdots&X_n^{q_n+1}\\
X_2^{r_2}&X_3^{r_3}&\cdots&X_1^{r_1}
\end{pmatrix}.
\]
In particular, the graded minimal $S$-free resolution of $k[H]$ is given by the Eagon--Northcott complex of the matrix mentioned above.
\end{theorem}

The interesting point is that the unique factorization of \(h\), together with the arithmetic structure of \(\PF(H)\), forces the relations
\[
(q_i+1)a_i+\alpha\in H
\]
to form a single directed cycle through all the minimal generators. This
produces the cyclic monomial matrix. The case \(h=0\) is compatible with the
determinantal situation studied in \cite{GKMT}; the theorem above applies to an arbitrary nonnegative translation \(h\). In particular, if $H$ has maximal embedding dimension, we recover the statement \cite[Theorem 2 $(2)\Rightarrow (4)$]{KM}.

We then apply this criterion to numerical semigroups generated by the orbit of
their multiplicity under an affine map. Let $a,n,c$ be integers such that $a\ge 2, n\ge 3, 1 \le c\le a.$ For each $1\le i\le n$, put $ A_i=1+a+\cdots+a^{i-1}$ and $m = cA_{n-1} + 1$. Let $b$ be an integers such that $b\ge -(a-2)m-2$ and $\gcd(m, b) = 1.$ Let  $d=(a-1)m+b$ and $s_i=m+dA_i$ for all $i$. Consider the numerical semigroup $H=\left<s_0,s_1,\ldots,s_{n-1}\right>$, where $s_0=m$. In \cite{AMO}, 
\'Alvarez, Moreno-\'Avila and Ojeda determined the
pseudo-Frobenius numbers
\[
\PF(H)=
\{\rho,\rho-b,\ldots,\rho-(n-2)b\},
\]
where $\rho=d(m-1)+(c-1)m.$ They also considered the monomial matrix
\[
X_c=
\begin{pmatrix}
X_1^a&X_2^a&\cdots&X_{n-1}^a&X_n^c\\
X_2&X_3&\cdots&X_n&X_1^{d+c+1-a}
\end{pmatrix}
\]
and asked whether its \(2\times2\) minors generate the whole defining ideal;
see \cite[Remark~8.2]{AMO}.

We answer this question affirmatively. Our proof of the determinantal
presentation is independent of their proof of the containment
\(\rmI_2(X_c)\subseteq I_H\). Using the weighted-degree criterion of
\cite[Section~2.3]{KMO}, that containment follows directly from the affine
relations among the generators. We then prove that, after writing the above
pseudo-Frobenius numbers in increasing order as
\[
\PF(H)=
\{h+\alpha,h+2\alpha,\ldots,h+(n-1)\alpha\},
\qquad \alpha=|b|,
\]
the element \(h\) has a unique factorization. Theorem
\ref{thm:intro-general} therefore yields the reverse containment.

We obtain the following result.
\begin{theorem}[Theorems \ref{thm:affine-unique} and \ref{thm:Xc}]\label{thm:intro-affine}
Let \(a\ge2\), \(n\ge3\), \(1\le c\le a\), $A_0=0 $ and $A_i=\dfrac{a^i-1}{a-1}$ for all $1\le i\le n$. Let $m=cA_{n-1}+1$ and \(b\in\ZZ\) satisfy $b\ge -(a-2)m-2$ and $\gcd(m, b) = 1.$ Put $d=(a-1)m+b, s_i=m+dA_i$, and \(H=\langle s_0,\ldots,s_{n-1}\rangle\). Then the following statements hold true.
\begin{enumerate}
\item[1)] $\PF(H)=\{h+\alpha,h+2\alpha,\ldots,h+(n-1)\alpha\}$ for some $h\ge 0, \alpha>0$.
\item[2)] $h$ has a unique factorization and
\[
I_H=\rmI_2\begin{pmatrix}
	X_1^a&X_2^a&\cdots&X_{n-1}^a&X_n^c\\
	X_2&x_3&\cdots&X_n&x_1^{d+c+1-a}
\end{pmatrix}.
\]
\end{enumerate} Consequently, the Eagon--Northcott complex associated to the matrix in above display gives the graded minimal free resolution of \(k[H]\).
\end{theorem}

For this explicit family, the equality of ideals can also be proved by reducing modulo $X_1$ and comparing lengths; see Remark \ref{rem:direct}. The application of Theorem \ref{thm:intro-general} has a different purpose. The arithmetic pseudo-Frobenius progression and the unique factorization of $h$ force
the relations among the minimal generators to form a single cycle. That cycle determines the exponents of a determinantal matrix, and the theorem then identifies its minors with the
defining ideal. Thus the presentation of this family is an instance of a criterion formulated without reference to the orbit parameters $a, b, c, m$ or to an a priori candidate matrix. The length computation proves the equality for $X_c$ efficiently, whereas the criterion explains why
the pseudo-Frobenius calculation of \cite{AMO} leads to a determinantal presentation and provides a route for other families with the same arithmetic and factorization properties.

The paper is organized as follows. Section~2 recalls the basic facts on
numerical semigroups, pseudo-Frobenius numbers, and the weighted-degree
criterion used for binomial relations. Section~3 develops the
unique-factorization criterion and proves Theorem~\ref{thm:intro-general}.
Section~4 applies it to the affine-closed family and proves
Theorem~\ref{thm:intro-affine}.

\section{Preliminaries}

Throughout this paper, let $H=\langle a_1,\ldots,a_n\rangle$
be a numerical semigroup minimally generated by \(a_1,\ldots,a_n\). Let \(k\) be a field and $R=k[H]=k[t^{a_1},\ldots,t^{a_n}]$
associated to \(H\). Let $S=k[X_1,\ldots,X_n]$ be the weighted polynomial ring with $	\deg X_i=a_i$ for every $1\le i\le n$ and consider the graded \(k\)-algebra homomorphism $\varphi:S\longrightarrow R$ given by
$\varphi(X_i)=t^{a_i}$ for all $i$. Recall that the defining ideal $I_H$ of $R$ is defined to be the kernel of $\varphi$.
\begin{definition}
For \(0<h\in H\), the Ap\'ery set of \(H\) with respect to \(h\) is defined by
\[
\Ap(H,h)=\{a\in H:a-h\notin H\}.
\]
Equivalently, if
\[
w_i=\min\{a\in H\mid a\equiv i\pmod h\} \text{ for } 0\le \forall i\le h-1,
\]
then
\[
\Ap(H,h)=\{w_0=0,w_1,\ldots,w_{h-1}\}.
\]
\end{definition}

\begin{definition}
An integer \(f\in\ZZ\setminus H\) is called a pseudo-Frobenius number of \(H\) if $f+s\in H$ for every $0<s\in H$. The set of all pseudo-Frobenius numbers is denoted by \(\PF(H)\).
\end{definition}

We recall the following two facts.
\begin{lemma}[{\cite[Proposition~2.19(2)]{RGS}}]\label{lem:PF-separation}
For \(z\in\ZZ\), one has \(z\notin H\) if and only if there exists
\(f\in\PF(H)\) such that \(f-z\in H\).
\end{lemma}

\begin{lemma}[{\cite[Proposition~2.20]{RGS}}]\label{lem:PF-Apery}
For each \(0<h\in H\),
\[
\PF(H)=
\{a-h\mid a\in\Ap(H,h),\ b-a\notin H
\text{ for all }b\in\Ap(H,h)\setminus\{a\}\}.
\]
\end{lemma}

The following elementary criterion will be used repeatedly. It is the
weighted-degree argument of \cite[Lemma~2.11 and Corollary~2.12]{KMO}.

\begin{lemma}[Weighted-degree criterion]\label{lem:degree}
Let the matrix
\[
M=
\begin{pmatrix}
f_1&f_2&\cdots&f_n\\
g_1&g_2&\cdots&g_n
\end{pmatrix},
\]
where \(f_i,g_i\) are monomials of positive degree in \(S\). Suppose that there
exists an integer \(\delta\) such that $\deg g_i-\deg f_i=\delta$ for all $1\le i\le n$. Then $\rmI_2(M)\subseteq I_H.$
\end{lemma}

\begin{proof}
For \(1\le i<j\le n\), set $\Delta_{ij}=f_i g_j-f_jg_i$.
By assumption, we have
\[
\deg f_i+\deg g_j
=\deg f_i+\deg f_j+\delta
=\deg f_j+\deg g_i.
\]
Every monomial \(u\in S\) satisfies $\varphi_H(u)=t^{\deg u}$.
Hence \(\varphi_H(\Delta_{ij})=0\), so every \(2\times2\) minor of \(M\) belongs to \(I_H\).
\end{proof}

\begin{remark}
The integer \(\delta\) in Lemma~\ref{lem:degree} is not required to be positive. This will be useful in the last section for the affine family when \(b<0\).
\end{remark}

\section{A unique-factorization criterion for determinantal defining ideals}

Throughout this section, let $H=\langle a_1,\ldots,a_n\rangle, n\ge3$,
and assume that
\begin{equation}\tag{$\star$}
\PF(H)=
\{h+\alpha,h+2\alpha,\ldots,h+(n-1)\alpha\},
\end{equation}
where \(h\ge0\) and \(\alpha>0\).

\begin{setting}\label{setting:mi}
For each \(1\le i\le n\), put $m_i=\min\{\ell>0\mid\ell a_i+\alpha\in H\},
 M_i=m_i-1$, and define
\[
h'=M_1 a_1+M_2 a_2+\cdots+M_n a_n.
\]
We also set
\[
\Zfac(h)=
\left\{(q_1,\ldots,q_n)\in\NN^n\mid
h=\sum_{i=1}^nq_i a_i\right\}.
\]
\end{setting}

Let us begin with a basic observation.

\begin{lemma}[{\cite[Lemma~2]{KM}}]\label{lem:h-alpha}
Under the condition ($\star$), one has $h\in H$ and $\alpha\notin H$.
\end{lemma}

\begin{proof}
If \(\alpha\in H\), then $(h+\alpha)+\alpha=h+2\alpha\in H$,
because \(h+\alpha\in\PF(H)\), a contradiction. Hence \(\alpha\notin H\). By Lemma~\ref{lem:PF-separation}, there exists \(1\le i\le n-1\) such that
\[
h+i\alpha-\alpha=h+(i-1)\alpha\in H.
\]
If \(i\ge2\), then \(h+(i-1)\alpha\in\PF(H)\) which is a contradiction. Therefore
\(i=1\), and \(h\in H\).
\end{proof}

\begin{lemma}[Boundary lemma]\label{lem:boundary}
For every \(c\in H\), then $c+\alpha\notin H$ if and only if $h-c\in H.$
\end{lemma}

\begin{proof}
Suppose first that \(c+\alpha\notin H\). By
Lemma~\ref{lem:PF-separation}, there exists \(1\le i\le n-1\) such that
\[
h+i\alpha-(c+\alpha)
=h+(i-1)\alpha-c\in H.
\]
If \(i\ge2\), then adding \(c\in H\) gives
\(h+(i-1)\alpha\in H\), contradicting ($\star$). Thus \(i=1\),
and \(h-c\in H\).

Conversely, suppose that \(h-c\in H\) and \(c+\alpha\in H\). Then $h+\alpha=(h-c)+(c+\alpha)\in H$, contrary to \(h+\alpha\in\PF(H)\).
\end{proof}

\begin{proposition}\label{prop:Mi}
For every \(i\), we have
\[
M_i=
\max\{q_i\mid(q_1,\ldots,q_n)\in\Zfac(h)\}.
\]
\end{proposition}

\begin{proof}
By definition, one has $M_i a_i+\alpha\notin H$ and $(M_i+1)a_i+\alpha\in H$. Applying Lemma~\ref{lem:boundary}, we obtain
$h-M_i a_i\in H$ and $h-(M_i+1)a_i\notin H$. Hence \(M_i\) is exactly the maximum possible \(i\)-th coordinate among the
factorizations of \(h\).
\end{proof}

\begin{corollary}\label{cor:hprime}
One has $h'-h\in H$.
\end{corollary}

\begin{proof}
Take any factorization $h=\mathop\sum\limits_{i=1}^nq_i a_i$ with $0\le q_i\in \ZZ$ for all $i$. We then have by Proposition~\ref{prop:Mi} that \(q_i\le M_i\) for all \(i\). Therefore
\[
h'-h=\sum_{i=1}^n(M_i-q_i)a_i\in H.
\]
\end{proof}

\begin{theorem}\label{thm:equiv}
Under ($\star$), the following statements are equivalent.
\begin{enumerate}
\item[1)] \(h\) has a unique factorization in \(H\).
\item[2)] \(h'=h\).
\item[3)] \(h'+\alpha\notin H\).
\end{enumerate}
\end{theorem}

\begin{proof}
Suppose first that $h=q_1a_1+\cdots+q_na_n$
is the unique factorization of \(h\). By Proposition~\ref{prop:Mi}, for each
\(i\) there exists a factorization of \(h\) whose \(i\)-th coordinate is
\(M_i\). By uniqueness, this factorization must be the one above. Hence
\(q_i=M_i\) for all \(i\), so \(h=h'\).

Conversely, suppose \(h=h'\), and let $h=q_1a_1+\cdots+q_na_n$
be any factorization. Then we have by Proposition~\ref{prop:Mi} that \(q_i\le M_i\), and therefore
\[
0=h'-h=\sum_{i=1}^n(M_i-q_i)a_i.
\]
Since all \(a_i>0\), it follows that \(q_i=M_i\) for every \(i\). Thus the
factorization is unique.

Finally, Lemma~\ref{lem:boundary}, applied to \(c=h'\), gives the equivalence
\[
h'+\alpha\notin H
\quad\Longleftrightarrow\quad
h-h'\in H.
\]
Moreover, we already know by Corollary~\ref{cor:hprime} that \(h'-h\in H\). Since
\(H\subseteq\NN\), both \(h-h'\) and \(h'-h\) can belong to \(H\) only when
\(h=h'\).
\end{proof}

We now prove the main result of this section.

\begin{theorem}\label{thm:general}
Assume ($\star$). Suppose $h=q_1a_1+\cdots+q_na_n$ is the unique factorization of \(h\). Then, after a suitable permutation of
\(a_1,\ldots,a_n\), there exist positive integers
\(r_1,\ldots,r_n\) such that
\[
(q_i+1)a_i+\alpha=r_{i+1}a_{i+1}, 
1\le\forall i\le n,
\]
where indices are taken modulo \(n\). Consequently,
\[
I_H=
\rmI_2\begin{pmatrix}
X_1^{q_1+1}&X_2^{q_2+1}&\cdots&X_n^{q_n+1}\\
X_2^{r_2}&X_3^{r_3}&\cdots&X_1^{r_1}
\end{pmatrix}.
\]
In particular, the graded minimal $S$-free resolution of $k[H]$ is given by the Eagon--Northcott complex \cite{EN} of the the matrix mentioned above.
\end{theorem}

\begin{proof}
Since the factorization of \(h\) is unique,
Proposition~\ref{prop:Mi} gives \(q_i=M_i\) for all \(i\). Thus
$q_i a_i+\alpha\notin H$ and $(q_i+1)a_i+\alpha\in H$.
Put $c_i=(q_i+1)a_i+\alpha$. No factorization of \(c_i\) can contain \(a_i\), because subtracting one copy
of \(a_i\) would imply \(q_i a_i+\alpha\in H\).

For each \(i\), we define
\[
T_i=\{j\mid c_i-a_j\in H, 1\le j\le n\}.
\]
Since \(c_i\in H\setminus\{0\}\), the set \(T_i\) is nonempty, while
\(i\notin T_i\). Choose \(\sigma(i)\in T_i\) for every \(i\). Then
$c_i=a_{\sigma(i)}+u_i$ for some $u_i\in H$.

Consider the directed graph with the vertex set $\{1,2,...,n\}$ and directed edges from vertex $i$ to vertex $\sigma(i)$. It contains a directed cycle \(C\), say of length \(r\). Summing the relations
corresponding to the vertices of \(C\) gives
\[
\sum_{i\in C}(q_i+1)a_i+r\alpha
=
\sum_{i\in C}a_{\sigma(i)}+\sum_{i\in C}u_i.
\]
Since \(\sigma\) cyclically permutes the vertices of \(C\),
\[
\sum_{i\in C}a_{\sigma(i)}=\sum_{i\in C}a_i.
\]
Hence
\[
\sum_{i\in C}q_i a_i+r\alpha\in H.
\]
Adding \(\mathop\sum\limits_{i\notin C}q_i a_i\in H\), we obtain $h+r\alpha\in H$.

If \(r<n\), then \(1\le r\le n-1\), contradicting
\(h+r\alpha\in\PF(H)\). Therefore every possible choice of the arrows produces a single directed \(n\)-cycle. We claim that each \(T_i\) is a singleton. Fix one such \(n\)-cycle. If \(T_i\) contained another vertex different from the successor of \(i\) in the
cycle, replacing only the outgoing arrow from \(i\) by this new arrow would
produce a directed cycle of length strictly smaller than \(n\), a
contradiction. Thus \(T_i=\{\sigma(i)\}\).

It follows that every factorization of \(c_i\) uses only the generator
\(a_{\sigma(i)}\). Hence $c_i=r_{\sigma(i)}a_{\sigma(i)}$ for some positive integer \(r_{\sigma(i)}\). After relabelling the generators
along the \(n\)-cycle, we obtain
\[
(q_i+1)a_i+\alpha=r_{i+1}a_{i+1} \text{ for } 1\le \forall i\le n.
\]

Set
\[
J=
\rmI_2\begin{pmatrix}
X_1^{q_1+1}&X_2^{q_2+1}&\cdots&X_n^{q_n+1}\\
X_2^{r_2}&X_3^{r_3}&\cdots&X_1^{r_1}
\end{pmatrix}.
\]
Then the degree difference between the entries in each column is constant. More precisely,
\[
\deg X_{i+1}^{r_{i+1}}
-\deg X_i^{q_i+1}
=
r_{i+1}a_{i+1}-(q_i+1)a_i
=
\alpha.
\]
Hence by Lemma~\ref{lem:degree} we get the inclusion $J\subseteq I_H$.

We now prove equality. First, oberse that
\[
\sqrt{J+(X_1)}=(X_1,\ldots,X_n).
\]
Indeed, modulo \(X_1\), the minor of the first two columns gives
\(X_2^{q_2+1+r_2}\). Hence \(X_2\in\sqrt{J+(X_1)}\). Inductively, once
\(X_i\in\sqrt{J+(X_1)}\), the minor of columns \(i\) and \(i+1\) shows that
\[
X_{i+1}^{q_{i+1}+1+r_{i+1}}\in\sqrt{J+(X_1)}.
\]
Thus every \(X_i\) belongs to the radical. Hence $
\htop(J+(X_1))=n$.
By Krull's principal ideal theorem, $\htop J\ge n-1$. 

On the other hand, the general height bound for the ideal of maximal minors of a \(2\times n\) matrix gives $\htop J\le n-1$ (\cite[Theorem 2.1]{BV88}). Therefore $\htop J=n-1$. The Eagon--Northcott complex is consequently a graded minimal free resolution of $A=S/J$. 

Put \[A_0=\sum_{i=1}^n a_i,\quad
d_i=(q_i+1)a_i\text{ and, }
D=\sum_{i=1}^n d_i.
\]
Since \(h=\mathop\sum\limits_iq_i a_i\), we have $D=h+A_0$. After interchanging the two rows of the matrix defining \(J\), the entries in
the \(i\)-th column have degrees \(d_i+\alpha\) and \(d_i\). Hence the last
free module in the graded Eagon--Northcott resolution is
\[
\bigoplus_{j=1}^{n-1}S(-D-j\alpha).
\]
Since the graded canonical module of \(S\) is \(S(-A_0)\), dualizing this
resolution shows that the canonical module \(\mathrm K_A\) of $A$ is minimally generated in the degrees
\[
-h-\alpha,\ -h-2\alpha,\ \ldots,\ -h-(n-1)\alpha.
\]
By the assumption ($\star$) and the standard description of the canonical module of a numerical semigroup ring \cite{GW}, the canonical module \(\mathrm K_{k[H]}\) of $k[H]$ has one minimal generator in each of exactly the same
degrees.

Let $L=I_H/J$. If \(L=0\), there is nothing to prove. Assume \(L\ne0\). Since \(A\) is a
one-dimensional Cohen--Macaulay ring and \(L\subseteq A\), one has
\(H^0_{\mathfrak m}(L)=0\), where
\(\mathfrak m=(X_1,\ldots,X_n)\). Thus \(L\) is a one-dimensional
Cohen--Macaulay \(S\)-module. Applying
\(\operatorname{Ext}^{n-1}_S(-,S(-A_0))\) to the exact sequence
\[
0\longrightarrow L\longrightarrow A\longrightarrow k[H]\longrightarrow0
\]
therefore yields a short exact sequence
\[
0\longrightarrow
\mathrm K_{k[H]}
\xrightarrow{\ \varepsilon\ }
\mathrm K_A
\longrightarrow
\mathrm K_L
\longrightarrow0.
\]

For \(1\le r\le n-2\), one has $r\alpha\notin H$. Indeed, if \(r\alpha\in H\), then, since \(h+\alpha\in\PF(H)\),
\[
h+(r+1)\alpha=(h+\alpha)+r\alpha\in H,
\]
contradicting ($\star$). Consequently, if
\(1\le i<j\le n-1\), the difference
\[
(-h-i\alpha)-(-h-j\alpha)=(j-i)\alpha
\]
does not belong to \(H\). Hence, in each degree
\(-h-i\alpha\), no minimal generator of either canonical module can be
obtained from another one by multiplication with a positive-degree element
of \(S\).

It follows that the map induced by \(\varepsilon\) on
\[
\mathrm K_{k[H]}/\mathfrak m\mathrm K_{k[H]}
\longrightarrow
\mathrm K_A/\mathfrak m\mathrm K_A
\]
is, in each of the \(n-1\) generator degrees, a nonzero map between
one-dimensional \(k\)-vector spaces. Thus it is an isomorphism. By graded
Nakayama's lemma, \(\varepsilon\) is surjective and hence an isomorphism.
Therefore $\mathrm K_L=0$.
This is impossible for a nonzero one-dimensional Cohen--Macaulay module.
Hence \(L=0\), and therefore $J=I_H$.
\end{proof}
\begin{corollary}\label{cor38}
Assume ($\star$). Then $h$ has a unique factorization in
	$H$ if and only if $h+n\alpha$ has a unique factorization in $H$. When these
	conditions hold, the determinantal conclusion of Theorem \ref{thm:general} holds.
\end{corollary}
\begin{proof}
We first use an argument similar to that in the first part of the proof of Theorem \ref{thm:general} to establish a factorization result covering both possible signs of the common difference.
\begin{claim}
Suppose $	\PF(H)=\{g+\beta,g+2\beta,\ldots,g+(n-1)\beta\}$ for $\beta\in\mathbb Z\setminus\{0\}$, and $g=\mathop\sum\limits_{i=1}^nq_i a_i$ is the unique factorization of $g$ in $H$. Then $g+n\beta$ also has a unique factorization.
\end{claim}
\noindent{\it Proof of Claim}. For every $c\in H$, the argument of Boundary lemma (Lemma 3.3) works for either sign
	of $\beta$ and gives the equivalence
	\begin{equation*}\label{eq:boundary-both-signs}
		c+\beta\notin H\quad\Longleftrightarrow\quad g-c\in H.
	\end{equation*}
	Indeed, if $c+\beta\notin H$, Lemma 2.3 gives an index
	$1\le j\le n-1$ for which
	$g+(j-1)\beta-c\in H$. If $j\ge2$, adding $c$ would put
	$g+(j-1)\beta\in\PF(H)$ in $H$, a contradiction. Hence $j=1$ so that $g-c\in H$. 
	Conversely, $g-c\in H$ and $c+\beta\in H$ would imply
	$g+\beta\in H$ which is impossible.
	
	By the unique representation of $g$, one has $g-(q_i+1)a_i\notin H$, whereas $g-q_i a_i\in H$.
	Thus  the above equivalence shows that
	$c_i:=(q_i+1)a_i+\beta\in H$ and $q_i a_i+\beta\notin H$.
	In fact $c_i\ne0$. If otherwise, $-\beta=(q_i+1)a_i\in H\setminus\{0\}$ which implies $g+\beta=(g+2\beta)-\beta\in H$ a contradiction. No factorization of $c_i$ uses
	$a_i$, since $c_i-a_i=q_i a_i+\beta\notin H$.
	
	Choose a generator $a_{\sigma(i)}$ in a factorization of each $c_i$.
	The directed graph given by $i\mapsto\sigma(i)$ has a cycle $C$, say of length $r$. Writing $c_i=a_{\sigma(i)}+u_i$ for $i\in C$,
	and summing around this cycle, we obtain
	\[
	\sum_{i\in C}q_i a_i+r\beta=\sum_{i\in C}u_i\in H.
	\]
	Adding $\mathop\sum\limits_{i\notin C}q_i a_i$ gives $g+r\beta\in H$.
If $r<n$, this contradicts $g+r\beta\in\PF(H)$. Therefore every choice of the arrows gives one cycle containing all $n$ vertices.
Consequently, each $c_i$ must involve only the generator at its successor, since choosing a different successor at any vertex would create a shorter cycle. After relabelling, there are positive integers
	$v_1,\ldots,v_n$ such that
	\begin{equation*}
		(q_i+1)a_i+\beta=v_{i+1}a_{i+1},\,  1\le\forall i\le n,
	\end{equation*}
with indices taken modulo \(n\). Summing both sides, we obtain
	$g+n\beta=\mathop\sum\limits_{i=1}^n (v_i-1)a_i$. Now if
	$g+n\beta=\mathop\sum\limits_{i=1}^n w_i a_i$ is another factorization, then
	$w_i\ge v_i$ for some $i$. Hence
	\[
	g+(n-1)\beta
	=\bigl(g+n\beta-v_i a_i\bigr)+(q_{i-1}+1)a_{i-1}\in H,
	\]
	contrary to $g+(n-1)\beta\in\PF(H)$. This proves Claim. \qed

We now apply Claim with $(g,\beta)=(h,\alpha)$ to obtain
	that uniqueness of $h$ implies uniqueness of $h+n\alpha$.
	Conversely, by writing
	\[
	\PF(H)=\{(h+n\alpha)-\alpha,(h+n\alpha)-2\alpha,\ldots,
	(h+n\alpha)-(n-1)\alpha\}
	\]
and applying the Claim again with $(g,\beta)=(h+n\alpha,-\alpha)$, we see that
	uniqueness of $h+n\alpha$ implies uniqueness of
	$h+n\alpha+n(-\alpha)=h$. The final assertion follows from Theorem \ref{thm:general}.
\end{proof}
As a direct consequence, we recover the implication $(2)\Rightarrow (4)$ of \cite[Theorem 2]{KM} in the case where $H$ has maximal embedding dimension.
\begin{corollary}[\cite{KM}, Theorem 2]
Let $H=\left<a_1,a_2,\ldots,a_n\right>$ be a numerical semigroup minimally generated by $n$ elements $a_1<a_2<\cdots<a_n$. Suppose that $H$ has maximal embedding dimension, that is, $a_1=n$ and assume $\PF(H)=\{h+\alpha,h+2\alpha,\ldots,h+(n-1)\alpha\}$. Then, there exists a positive integer \(s\) such that
\[
I_H=
\rmI_2\begin{pmatrix}
	X_1^{s}&X_2&\cdots&X_{n-1}&X_{n}\\
	X_2&X_3&\cdots&X_n&X_1^{s+\alpha}
\end{pmatrix}.
\]
\end{corollary}
\begin{proof}
Since $H$ has maximal embedding dimension, $\PF(H)=\{a_2-a_1,\ldots,a_n-a_1\}$. It implies that $h+(i-1)\alpha=a_i-a_1$ for every $i=2,\ldots,n$. Since $0\le h < a_2$ and $h\in H$, every factorization of $h$ uses
only $a_1$; this is because all other minimal generators are at least $a_2$. Hence $h = (s-1)a_1$ for a uniquely determined integer $s\ge 1$, and this is the unique factorization of $h$. Moreover, there are equalities $$sa_1+\alpha=a_2,\, a_{i-1}+\alpha=a_i\text{ for all } 2\le i\le n, \text{ and } a_n+\alpha=(s+\alpha)a_1.$$ 
The conclusion therefore follows from Theorem \ref{thm:general}.

\end{proof}

\section{Affine-closed numerical semigroups}

Let \(a\ge2\), \(m>1\), and \(b\in\ZZ\) with \(\gcd(m,b)=1\) and $b\ge -(a-2)m-2$. We consider the affine map $T(z)=az+b$.
Put $A_0=0 \text{ and }$
\[A_i=\frac{a^i-1}{a-1}=1+a+\cdots+a^{i-1} \text{ for all } i>0.
\]
Set $d=(a-1)m+b$. The orbit of $m$ under $T$ is given by $s_i=T^i(m)=m+dA_i$.

In this section, we set $m=cA_{n-1}+1$ for $1\le c\le a$ and focus our attention on the semigroup $
H=\langle s_0,s_1,\ldots,s_{n-1}\rangle$ generated by the orbit of $m$ under the affine map $T$, where $s_0 = m$. Under this hypotheses, the semigroup $
H$ is minimally generated by \(n\) elements. Moreover,
\cite[Theorem~8.1]{AMO} provides the pseudo-Frobenius numbers
\begin{equation}\label{eq:PF-affine}
\PF(H)=
\{\rho,\rho-b,\rho-2b,\ldots,\rho-(n-2)b\},
\end{equation}
where $\rho=d(m-1)+(c-1)m$.

We shall use the elementary identity
\begin{equation}\label{eq:Ai-sum}
(a-1)\sum_{i=1}^{n-2}A_i
=A_{n-1}-(n-1).
\end{equation}
Indeed, \((a-1)A_i=a^i-1\), and summing yields
\[
(a-1)\sum_{i=1}^{n-2}A_i
=\sum_{i=1}^{n-2}(a^i-1)
=A_{n-1}-(n-1).
\]

\subsection{The case \(b>0\)}

Assume \(b>0\). Then \eqref{eq:PF-affine}, written in increasing order, is
\[
\PF(H)=
\{h+b,h+2b,\ldots,h+(n-1)b\},
\]
where $h=\rho-(n-1)b$.

\begin{theorem}\label{thm:bpositive}
Assume \(b>0\). Then
\begin{equation}\label{eq:h-positive}
h=(a-1)(s_0+s_1+\cdots+s_{n-2})+(c-1)s_{n-1}.
\end{equation}
Moreover, this is the unique factorization of \(h\).
\end{theorem}

\begin{proof}
Since \(m-1=cA_{n-1}\), the equality \eqref{eq:Ai-sum} gives
\[
(a-1)\sum_{i=1}^{n-2}A_i+(c-1)A_{n-1}
=cA_{n-1}-(n-1)=m-n.
\]
Hence the right-hand side of \eqref{eq:h-positive} is
\[
\bigl((a-1)(n-1)+c-1\bigr)m+d(m-n).
\]
Since \(b=d-(a-1)m\), this is equal to
\[
d(m-1)+(c-1)m-(n-1)b=h.
\]

Now suppose $h=\mathop\sum\limits_{i=0}^{n-1}x_i s_i$ is any factorization of $h$ in $H$. Put $L=\mathop\sum\limits_{i=0}^{n-1}x_i$ and $W=\mathop\sum\limits_{i=1}^{n-1}x_i A_i$. Then \(h=Lm+dW\). For the factorization \eqref{eq:h-positive}, put $L_0=(a-1)(n-1)+c-1$
 and $W_0=m-n$. We get $h=L_0m+dW_0$. Consequently, \[
(L-L_0)m+d(W-W_0)=0.
\]
Since $\gcd(m,d)=\gcd(m,b)=1$, there exists \(t\in\ZZ\) such that $W=W_0+tm$ and $L=L_0-td$.
Because \(0\le W_0=m-n<m\) and \(W\ge0\), we must have \(t\ge0\).

If \(t\ge1\), then $L\le L_0-d$. Since \(b>0\),
\[
d=(a-1)m+b>(a-1)m.
\]
Moreover \(m=cA_{n-1}+1\ge n\), while
\[
L_0=(a-1)(n-1)+c-1\le(a-1)n.
\]
Hence \(d>L_0\), which would imply \(L<0\), a contradiction. Therefore
\(t=0\) which implies that $L=L_0$ and $W=W_0$. 

Put $\ell=\mathop\sum\limits_{i=1}^{n-1}x_i$  and $\ell_0=(a-1)(n-2)+c-1$. Then the vector $q=(a-1,\ldots,a-1,c-1)$ satisfies $W_0=\mathop\sum\limits_{i=1}^{n-1}q_iA_i$. Since \(W_0=m-n<A_n\), Proposition~4.3 and Lemma~4.4 of \cite{AMO} show that this canonical representation has minimum length among all representations of \(W_0\) in terms of
\(A_1,\ldots,A_{n-1}\). Therefore $\ell\ge\ell_0$.

On the other hand, one has $x_0=L_0-\ell\ge0$, whence $\ell\le L_0=\ell_0+(a-1)$. Since $\mathop\sum\limits_{i=1}^{n-1}(x_i-q_i)A_i=0$,
multiplying by \(a-1\) and using \((a-1)A_i=a^i-1\) gives
\[
\sum_{i=1}^{n-1}(x_i-q_i)a^i=\ell-\ell_0.
\]
The left-hand side is divisible by \(a\). Hence $\ell-\ell_0\equiv0\pmod a$. Together with $0\le\ell-\ell_0\le a-1$,
this yields \(\ell=\ell_0\). Thus \(x_0=a-1\).

It remains to show that the canonical representation $q = (a-1,\ldots, a-1, c-1)$ of $W_0 = m-n$ is the only representation of length $\ell_0$. Indeed, let $x = (x_1,\ldots, x_{n-1})$ be such a representation. Apply the reduction procedure in the proof of \cite[Lemma 4.4]{AMO} to $x$. It terminates at the canonical representation $q$; every step weakly decreases total length. Since $\ell_0$ is the minimum length, every step in this sequence must preserve length. In a length-preserving step, \cite[Lemma 4.4]{AMO} has $k > 1$; it increases coordinate $k-1$ by $a$ and also increases a coordinate with higher index. Consequently, the output of this step cannot be a canonical representation; because a canonical vector cannot have a coordinate at least $a$ while a higher coordinate is positive. Thus no length-preserving step can be the last step before $q$. The reduction sequence has no steps and $x = q$. Together with $x_0 = a-1$, this proves the claimed uniqueness.
\end{proof}

\subsection{The case \(b<0\)}

Assume \(b<0\), and put $\alpha=-b>0$. Then \eqref{eq:PF-affine} becomes
\[
\PF(H)=\{\rho,\rho+\alpha,\ldots,\rho+(n-2)\alpha\}.
\]
Thus $h=\rho-\alpha=\rho+b.$
\begin{theorem}\label{thm:bnegative}
Assume \(b<0\). Then $h=(d+c-a)s_0$ and this is the unique factorization of \(h\).
\end{theorem}

\begin{proof}
Since \(b=d-(a-1)m\), there are equalities
\[
\begin{aligned}
h
&=d(m-1)+(c-1)m+b\\
&=d(m-1)+(c-1)m+d-(a-1)m\\
&=(d+c-a)m.
\end{aligned}
\]
Since \(s_0=m\), there is an equality $h=(d+c-a)s_0$.

Suppose $h=x_0 s_0+x_1 s_1+\cdots+x_{n-1} s_{n-1}$. Put $L=\mathop\sum\limits_{i=0}^{n-1}x_i$ and $ W=\sum_{i=1}^{n-1}x_iA_i$. Then $(d+c-a)m=Lm+dW$, so $(d+c-a-L)m=dW$. Hence, since \(\gcd(m,d)=1\), there exists \(t\in\NN\) such that
$W=tm$ and $d+c-a-L=td$.
Thus $L=(1-t)d+c-a$. 

If \(t=1\), then \(L=c-a\le0\). When \(c<a\), this is impossible because
\(W=m>0\); when \(c=a\), one has \(L=0\) while \(W=m>0\), again impossible.
If \(t\ge2\), then \(L<0\), also impossible.

Thus \(t=0\), and hence \(W=0\). We then get $
x_1=\cdots=x_{n-1}=0 $ and $x_0=d+c-a$ so that $h=(d+c-a)s_0$ is the unique factorization as desired.
\end{proof}

Combining the two cases yields the following.

\begin{theorem}\label{thm:affine-unique}
Under the hypotheses above, write \eqref{eq:PF-affine} in increasing order as
\[
\PF(H)=
\{h+\alpha,h+2\alpha,\ldots,h+(n-1)\alpha\},
\qquad \alpha=|b|.
\]
Then \(h\) has a unique factorization in \(H\). Consequently,
Conjecture~\ref{conj:main} holds for \(H\).
\end{theorem}

\begin{proof}
If \(b>0\), the assertion follows from Theorem~\ref{thm:bpositive}. If
\(b<0\), it follows from Theorem~\ref{thm:bnegative}. The determinantal
assertion then follows from Theorem~\ref{thm:general}.
\end{proof}

We now determine the defining ideal explicitly. Let $S=k[X_1,\ldots,X_n]$ be the graded polynomial ring with $\deg X_i=s_{i-1}$ for all $1\le i\le n$, and set the matrix
\[
\X_c=
\begin{pmatrix}
X_1^a&X_2^a&\cdots&X_{n-1}^a&X_n^c\\
X_2&X_3&\cdots&X_n&X_1^{d+c+1-a}
\end{pmatrix}.
\]

We first verify directly, independently of the determinantal containment in
\cite[Remark~8.2]{AMO}, that $\rmI_2(\X_c)\subseteq I_H$. Since $
A_i=aA_{i-1}+1$,
we have, for \(1\le i\le n-1\),
\[
\begin{aligned}
s_i
&=m+dA_i\\
&=m+d(aA_{i-1}+1)\\
&=a(m+dA_{i-1})+d-(a-1)m\\
&=as_{i-1}+b.
\end{aligned}
\]
Consequently,
\begin{equation}\label{eq:affine-degree1}
\deg X_{i+1}-\deg X_i^a
=s_i-as_{i-1}=b
\qquad(1\le i\le n-1).
\end{equation}
Moreover, using \(m-1=cA_{n-1}\), we obtain
\[
\begin{aligned}
cs_{n-1}+b
&=cm+cdA_{n-1}+b\\
&=cm+d(m-1)+b\\
&=cm+dm-d+b\\
&=(d+c)m-(d-b)\\
&=(d+c+1-a)m.
\end{aligned}
\]
Hence
\begin{equation}\label{eq:affine-degree2}
\deg X_1^{d+c+1-a}-\deg X_n^c
=(d+c+1-a)s_0-cs_{n-1}=b.
\end{equation}
Thus the difference between the weighted degrees of the bottom and top entries
is the same, namely \(b\), in every column of \(\X_c\). By
Lemma~\ref{lem:degree}, we obtain the inclusion
\begin{equation}\label{eq:Xc-containment}
\rmI_2(\X_c)\subseteq I_H.
\end{equation}
Notice that this argument works uniformly for both \(b>0\) and \(b<0\).

\begin{theorem}\label{thm:Xc}
Under the hypotheses of Theorem~\ref{thm:affine-unique}, we have
$I_H=\rmI_2(\X_c)$. In particular, the Eagon--Northcott complex associated to \(\X_c\) is the graded minimal free resolution of \(k[H]\).
\end{theorem}

\begin{proof}
The containment $\rmI_2(\X_c)\subseteq I_H$ has already been proved in \eqref{eq:Xc-containment} by the weighted-degree
criterion. It remains to identify the matrix supplied by
Theorem~\ref{thm:general}. 

Suppose first that \(b>0\). Then
\(\alpha=b\), and Theorem~\ref{thm:bpositive} gives
\[
h=(a-1)s_0+\cdots+(a-1)s_{n-2}+(c-1)s_{n-1}.
\]
Thus the exponents in the first row of the matrix are $a,\ldots,a,c$. Moreover, the relations in \eqref{eq:affine-degree1} and \eqref{eq:affine-degree2} give
$as_{i-1}+b=s_i$, for all $1\le i\le n-1$, and
\[
cs_{n-1}+b=(d+c+1-a)s_0.
\]
Hence the second-row exponents are $1,\ldots,1,d+c+1-a$.
Therefore the matrix in Theorem~\ref{thm:general} is precisely \(X_c\).

Now suppose \(b<0\), and put \(\alpha=-b\). By
Theorem~\ref{thm:bnegative}, we have that the decomposition
\[
h=(d+c-a)s_0
\]
is the unique factorization of \(h\). The affine relations can be rewritten as
\[
s_i+\alpha=as_{i-1}
\qquad(1\le i\le n-1)
\]
and
\[
(d+c+1-a)s_0+\alpha=cs_{n-1}.
\]
Ordering the generators as $s_0,s_{n-1},s_{n-2},\ldots,s_1$.
Theorem~\ref{thm:general} yields the matrix
\[
\begin{pmatrix}
X_1^{d+c+1-a}&X_n&x_{n-1}&\cdots&X_2\\
X_n^c&X_{n-1}^a&X_{n-2}^a&\cdots&X_1^a
\end{pmatrix}.
\]
Interchanging its two rows and reversing the order of the columns gives
exactly \(\X_c\). Neither operation changes the ideal generated by the
\(2\times2\) minors. Hence $I_H=\rmI_2(\X_c)$ in both cases.

Finally, $\htop I_H=n-1$. Since every entry of \(\X_c\) has positive degree, the Eagon--Northcott gives the graded minimal free resolution of $k[H]$.
\end{proof}
\begin{remark}\label{rem:direct}
	For the explicit affine family, the determinantal equality can also be	checked directly. Let $J=I_2(X_c)\subseteq I_H$. For a monomial matrix
	\[
	M=\begin{pmatrix}
		X_1^{\ell_1}&X_2^{\ell_2}&\cdots&X_n^{\ell_n}\\
		X_2^{m_2}&X_3^{m_3}&\cdots&X_1^{m_1}
	\end{pmatrix}
	\]
	of height $n-1$, the Eagon--Northcott Hilbert series,
	or a count of the standard monomials modulo $X_1$, gives
	\[
	\dim_k S/(I_2(M)+(X_1))
	=\ell_2\cdots\ell_n+m_2\ell_3\cdots\ell_n+
	\cdots+m_2\cdots m_n.
	\]
	For $X_c$ we have $(\ell_2,\ldots,\ell_n)=(a,\ldots,a,c)$ and
	$(m_2,\ldots,m_n)=(1,\ldots,1)$, so this length is
	\[ca^{n-2}+ca^{n-3}+\cdots+ca+c+1
	=cA_{n-1}+1=m.\]
	On the other hand, $S/(I_H+(X_1))\cong k[H]/(t^m)$ has $k$-length $m$.
	The surjection between these Artinian quotients is therefore an
	isomorphism. If $I_H/J$ were nonzero, it would be a nonzero finitely
	generated graded module over $S/J$ and would have a nonzero quotient
	modulo $X_1$ by graded Nakayama, a contradiction. Thus $J=I_H$.
	
This proof uses the explicit exponents of $X_c$ and the identity $m=cA_{n-1}+1$ to obtain the required length. In contrast, Theorem \ref{thm:general} constructs the cyclic relations from the pseudo-Frobenius progression and the unique factorization of $h$, then supplies the ideal equality without a separate length calculation. The latter argument makes the structural reason for the determinantal presentation visible and is 	available whenever those two hypotheses can be verified.
\end{remark}

We close this paper with two examples satisfying the hypotheses.

\begin{example}[The case \(b>0\)]
Take \(a=2\), \(n=4\), and \(c=1\). Then $A_3=1+2+4=7$ and $m=A_3+1=8$.
Choose \(b=3\). Then $d=(a-1)m+b=11$ and
\[
H=\langle8,19,41,85\rangle.
\]
We have $\rho=11(8-1)=77$, and hence
\[
\PF(H)=\{77,74,71\}=\{71,74,77\}.
\]
Thus \(h=68\) and \(\alpha=3\). Theorem~\ref{thm:bpositive} gives the factorization $68=8+19+41$, and this is the unique factorization of \(68\). The defining ideal is
\[
I_H=
\rmI_2\begin{pmatrix}
X_1^2&X_2^2&X_3^2&X_4\\
X_2&X_3&X_4&X_1^{11}
\end{pmatrix}.
\]
\end{example}

\begin{example}[The case \(b<0\)]
Again take \(a=2\), \(n=4\), \(c=1\), and \(m=8\), but now let
\(b=-1\). Then $d=7$ and $H=\langle8,15,29,57\rangle$. We have $
\rho=7(8-1)=49$,
so
\[
\PF(H)=\{49,50,51\}.
\]
Hence \(h=48\) and \(\alpha=1\). Theorem~\ref{thm:bnegative} gives the factorization $48=6\cdot8$,
and this is the unique factorization of \(48\). The defining ideal is
\[
I_H=
\rmI_2\begin{pmatrix}
X_1^2&X_2^2&X_3^2&X_4\\
X_2&X_3&X_4&X_1^7
\end{pmatrix}.
\]
\end{example}
\section*{Ackknowledgements}
The author is grateful to Professor Naoyuki Matsuoka for insightful discussions on Conjecture 1.1 over the years. Furthermore, the author sincerely thanks Professors Naoyuki Matsuoka and Toshinori Kobayashi for providing the proofs of Corollary \ref{cor38} and Remark \ref{rem:direct}.

\end{document}